%% file: main_arxiv.tex
\documentclass[hidelinks,onefignum,onetabnum,nohypdvips]{siamart220329}

\input{ex_shared}

\headers{Unifying safe ball regions for safe screening}{}

\title{One to beat them all: ``RYU'' -- a unifying framework\\ for the construction of safe balls}

\author{
	Thu-Le Tran\thanks{
		Univ. Rennes, IRMAR - UMR 6625, F-35000 Rennes, France \\
		Mathematic department, School of Education, Can Tho Univ, Vietnam
	}
	\and Cl{\'e}ment Elvira\thanks{
		IETR UMR CNRS 6164, CentraleSupelec Rennes Campus, 35576 Cesson Sévigné, France
	}
	\and Hong-Phuong Dang\thanks{
		IETR UMR CNRS 6164, CentraleSupelec Rennes Campus, 35576 Cesson Sévigné, France
	}
	\and C{\'e}dric Herzet\thanks{
		Univ. Rennes, Ensai, CNRS, CREST - UMR 9194, F-35000 Rennes, France
	}
}

\usepackage{kmath}
\usepackage{enumitem}
\input{togles}
\input{headings_macros}

\newsiamthm{lemm}{Lemma}
\newsiamthm{theo}{Theorem}
\newsiamthm{prop}{Proposition}

\crefname{prop}{Proposition}{Propositions}

\newcommand{\hypref}[1]{\ref{#1}}

\settoggle{arxiv}{true}

\begin{document}

\maketitle

\begin{abstract}
  In this paper, we present a new framework, called “\Holder{}”, for constructing ``safe'' regions -- specifically, bounded sets that are guaranteed to contain the dual solution of a target optimization problem.
  Our framework applies to the standard case where the objective function is composed of two components: a closed, proper, convex function with Lipschitz-smooth gradient and another closed, proper, convex function.
  We show that the \Holder{} framework not only encompasses but also improves upon the state-of-the-art methods proposed over the past decade for this class of optimization problems.
\end{abstract}



\input{include_all}


\bibliographystyle{siamplain}

\end{document}

%% file: ex_shared.tex
\newsiamremark{remark}{Remark}
\newsiamremark{hypothesis}{Hypothesis}
\crefname{hypothesis}{Hypothesis}{Hypotheses}
\newsiamthm{claim}{Claim}

\headers{Unifying safe ball regions for safe screening}{}

\usepackage{amsopn}


%% file: togles.tex
\usepackage{etoolbox}
\providetoggle{techreport}
\settoggle{techreport}{false}

\providetoggle{cameraready}
\settoggle{cameraready}{true}

\providetoggle{arxiv}
\settoggle{arxiv}{true}

\providetoggle{trackref}
\settoggle{trackref}{false}

\providetoggle{pagebreakSection}
\settoggle{pagebreakSection}{false}

\providetoggle{shortreference}
\settoggle{shortreference}{false}

\providetoggle{acknowledgments}
\settoggle{acknowledgments}{false}

%% file: headings_macros.tex
\usepackage{pgffor}

\iftoggle{cameraready}{
}{
	\definecolor{myGreen}{RGB}{35,139,69}
	\definecolor{myOrange}{RGB}{217,72,1}

}

\foreach \x in {a,...,z}{
		\expandafter\xdef\csname bf\x \endcsname{\noexpand\ensuremath{\noexpand\mathbf{\x}}}
		\expandafter\xdef\csname bs\x \endcsname{\noexpand\ensuremath{\noexpand\boldsymbol{\x}}}
	}

\foreach \x in {A,...,Z}{
		\expandafter\xdef\csname bs\x \endcsname{\noexpand\ensuremath{\noexpand\boldsymbol{\x}}}
		\expandafter\xdef\csname bf\x \endcsname{\noexpand\ensuremath{\noexpand\mathbf{\x}}}
		\expandafter\xdef\csname bb\x \endcsname{\noexpand\ensuremath{\noexpand\mathbb{\x}}}
		\expandafter\xdef\csname ds\x \endcsname{\noexpand\ensuremath{\noexpand\mathds{\x}}}
		\expandafter\xdef\csname cal\x \endcsname{\noexpand\ensuremath{\noexpand\mathcal{\x}}}
	}

\newcommand{\vecZero}{{\mathbf 0}}

\iftoggle{shortreference}{
	\newcommand{\pageInRef}[1]{}

	\newcommand{\definitionInRef}[1]{Def.~#1}
	\newcommand{\theoremInRef}[1]{Th.~#1}
	\newcommand{\propositionInRef}[1]{Prop.~#1}
	\newcommand{\corollaryInRef}[1]{Cor.~#1}
	\newcommand{\lemmaInRef}[1]{Lem.~#1}
	\newcommand{\itemInRef}[1]{Item~#1}

}{
	\newcommand{\pageInRef}[1]{}

	\newcommand{\definitionInRef}[1]{Definition~#1}
	\newcommand{\theoremInRef}[1]{Theorem~#1}
	\newcommand{\propositionInRef}[1]{Proposition~#1}
	\newcommand{\corollaryInRef}[1]{Corollary~#1}
	\newcommand{\lemmaInRef}[1]{Lemma~#1}
	\newcommand{\itemInRef}[1]{Item~#1}

}

\newcommand{\LASSO}{LASSO\xspace}

\def\ie{\textit{i.e.},\xspace}
\def\eg{\textit{e.g.},\xspace}

\def\ie{\textit{i.e.},\xspace}
\def\eg{\textit{e.g.},\xspace}

\newcommand{\defeq}{\triangleq}

\newcommand{\aScalar}{z}
\newcommand{\aScaling}{t}

\newcommand{\aDimension}{d}
\newcommand{\aVariable}{\mathbf{\aScalar}}
\newcommand{\aDualVariable}{\aVariable^*}
\newcommand{\aSubgradient}{\mathbf{g}}
\newcommand{\aFunction}{h}

\newcommand{\kRbar}{\overline{\kR}}

\DeclareMathOperator{\Fen}{\normalfont{\texttt{Fen}}}

\DeclareMathOperator{\Breg}{\normalfont{\texttt{Breg}}}

\newcommand{\obs}{\bfy}
\newcommand{\obsdim}{m}

\newcommand{\dicosymb}{a}
\newcommand{\dicomat}{\mathbf{\MakeUppercase{\dicosymb}}}
\newcommand{\atom}{\mathbf{\dicosymb}}

\newcommand{\pvsymb}{x}
\newcommand{\pvel}{\pvsymb}
\newcommand{\pv}{\mathbf{\pvsymb}}
\newcommand{\pvdim}{n}

\newcommand{\pvopt}{\pv^\star}

\newcommand{\dvsymb}{u}
\newcommand{\dv}{\mathbf{\dvsymb}}
\newcommand{\dvopt}{\dv^\star}

\newcommand{\dvdim}{\obsdim}

\newcommand{\pfunc}{P}
\NewDocumentCommand{\sigFunc}{o}{ \IfValueTF{#1}{ \pfunc(#1) }{ \pfunc } }

\newcommand{\loss}{f}

\newcommand{\reg}{g}

\newcommand{\dfunc}{D}
\NewDocumentCommand{\obsFunc}{o}{ \IfValueTF{#1}{ \dfunc(#1) }{ \dfunc } }

\newcommand{\gapfunc}{\normalfont{\texttt{GAP}}}

\newcommand{\idxobs}{i}
\newcommand{\idxatom}{j}

\newcommand{\convCoeff}{\alpha}

\DeclareMathOperator{\domain}{dom}

\newkmacro\sign[1][{\cdot}]{\mathrm{sign}({#1})}

\newcommand{\norm}[1]{\lVert#1\rVert}
\newcommand{\normOne}[1]{\lVert#1\rVert_1}
\newcommand{\normTwo}[1]{\lVert#1\rVert_2}

\newcommand{\dotp}[2]{\langle #1 \mid #2\rangle}

\NewDocumentCommand{\subDiff}{o}{ \IfValueTF{#1}{ \partial #1 }{ \partial } }

\newcommand{\saferegion}{\mathcal{S}}
\newcommand{\mathtexttt}[1]{\normalfont{\texttt{#1}}}

\newcommand{\ballsymb}{b}

\newcommand{\safeball}{\mathcal{\MakeUppercase{\ballsymb}}}
\newcommand{\spherec}[1][]{ \IfValueTF{#1}{ \mathbf{c}_{\mathtexttt{#1}} }{ \mathbf{c} } }
\newcommand{\spherer}[1][]{ \IfValueTF{#1}{ r_{\mathtexttt{#1}} }{ r } }

\newcommand{\GAPball}{\safeball_{\textnormal{\texttt{GAP}}}}
\newcommand{\GAPballc}{\mathbf{c}_{\textnormal{\texttt{GAP}}}}
\newcommand{\GAPballr}{r_{\textnormal{\texttt{GAP}}}}

\newcommand{\xGAPballc}{\mathbf{c}_{\pv\textnormal{\texttt{-GAP}}}}
\newcommand{\xGAPballr}{r_{\pv\textnormal{\texttt{-GAP}}}}

\newcommand{\FBI}{\text{RYU}}
\newcommand{\Holder}{\FBI\xspace}
\newcommand{\HOLDERball}{\safeball_{\textnormal{\texttt{\FBI}}}}
\newcommand{\HOLDERballc}{\mathbf{c}_{\textnormal{\texttt{\FBI}}}}
\newcommand{\HOLDERballr}{r_{\textnormal{\texttt{\FBI}}}}

\newcommand{\FNEballc}{\spherec[FNE]}
\newcommand{\FNEballr}{\spherer[FNE]}

\newcommand{\SASVIballc}{\spherec[SASVI]}
\newcommand{\SASVIballr}{\spherer[SASVI]}

\newcommand{\SAFEballc}{\spherec[SAFE]}
\newcommand{\SAFEballr}{\spherer[SAFE]}

\newcommand{\dynEDPPballc}{\spherec[dyn.EDPP]}
\newcommand{\dynEDPPballr}{\spherer[dyn.EDPP]}
\newcommand{\dynEDPPballparam}{{t^\star}}

\newcommand{\seqFERballc}{\spherec[SFER]}
\newcommand{\seqFERballr}{\spherer[SFER]}

\newcommand{\seqSLOREballc}{\spherec[SLORES]}
\newcommand{\seqSLOREballr}{\spherer[SLORES]}

%% file: include_all.tex

\input{manuscript_1_intro}
\input{manuscript_2_notations}
\input{manuscript_3_framework}
\input{manuscript_4_FBI}
\input{manuscript_5_comparison}
\input{conclusion}

\appendix

\input{manuscript_A_convex_considerations}
\input{manuscript_B_proof_safeness}
\input{manuscript_C_proof_connection}

%% file: manuscript_1_intro.tex
\section{Introduction}\label{sec:introduction}
\subsection{Context and state of the art}

In this paper, we consider the following family of optimization problems:
\begin{equation*} \label{eq:primal problem (fuzzy version)}
	\tag{P}
		\mbox{find } \pvopt\in\kargmin_{\pv \in \kR^{\pvdim}} \
		\pfunc(\pv) \triangleq
		\loss(\dicomat\pv) + \reg(\pv)		
\end{equation*}
where \(\dicomat\in\kR^{\dvdim\times\pvdim}\) is some known matrix,  \(\kfuncdef{\loss}{\kR^\dvdim}{\kR}\), \(\kfuncdef{\reg}{\kR^\pvdim}{\kR\cup\{+\infty\}}\) are proper, closed, convex functions and \(\loss\) is ``\(\convCoeff^{-1}\)-Lipschitz smooth over \(\kR^\dvdim\)'', that is \(\loss\) is differentiable everywhere on \(\kR^\dvdim\) and its gradient obeys the following regularity condition for some positive scalar \(\convCoeff>0\):
\begin{align}
	\forall \aVariable,\aVariable'\in\kR^{\dvdim}:\ \|\nabla\loss(\aVariable)-\nabla\loss(\aVariable')\|_2\leq \convCoeff^{-1} \|\aVariable-\aVariable'\|_2.
\end{align}
We assume moreover that~\eqref{eq:primal problem (fuzzy version)} is well-posed in the sense that at least one minimizer \(\pvopt\) exists.
Instances of problems satisfying these hypotheses are common in the literature of machine learning, statistics or signal processing,
and include (among many others) least-squares sparse regression~\cite{Chen1998}, logistic sparse regression~\cite{Koh2007AnIM} or the ``Elastic Net'' problem~\cite{Zou2005}.

The focus of this paper is on the construction of ``safe regions'', \ie 
subsets of \(\kR^\dvdim\) provably containing the unique solution
of the dual problem of \eqref{eq:primal problem (fuzzy version)}.
More specifically, our goal is to identify  some subset \(\saferegion\subseteq\kR^\dvdim\) such that $\dvopt\in\saferegion$ 
where
\begin{equation*} \label{def:dual solution}
	\tag{D}
	\dvopt = \kargmax_{\dv \in \kR^{\dvdim}} \
	\dfunc(\dv) \triangleq
	- \loss^*(-\dv) - \reg^*(\ktranspose{\dicomat}\dv)
\end{equation*}
and \(\loss^*\), \(\reg^*\) denote the convex conjugates of \(\loss\), \(\reg\).

The construction of safe regions has become an active field of research during the last decade (see \textit{e.g.}, \cite{Dai2012Ellipsoid,fercoq2015mind, Herzet16Screening,Herzet2022,ndiaye2017gap,pan2021safe,Tran2022, wang2022sequential,wang2015lasso,wang2014safe,yamada2021dynamic}) and has been triggered by the so-called ``safe feature elimination'' technique (also referred to as ``safe screening''), a procedure to accelerate the resolution of \eqref{eq:primal problem (fuzzy version)}, first  proposed in \cite{Laurent-El-Ghaoui:2012qs} and further extended in many contributions, see \textit{e.g.}, \cite{elvira2020safe,elvira2020icassp,guyard2022screen,ndiaye2021}.
One central element in the effectiveness of these acceleration methods is the identification (preferably at low computational cost) of \textit{small} safe regions with some specific geometry (\eg ball, ellipsoid, dome, etc).
In this paper, we focus on safe regions having a ``ball'' geometry, that is
\begin{equation} \label{eq:def ball}
	\saferegion
	=
	\safeball(\spherec,\spherer)
	\defeq \kset{
		\dv\in\kR^\dvdim
	}{
		\Vert \dv - \spherec \Vert_2 \leq \spherer
	}
\end{equation}
for some \(\spherec\in\kR^m\) and \(\spherer>0\).
In this respect, the state-of-the-art safe ball for the general family of optimization problems considered in this paper is indubitably the so-called ``GAP ball'' proposed in \cite{fercoq2015mind, ndiaye2017gap}.
It is defined for any couple $(\pv,\dv)
	\in \domain(\pfunc)\times\domain(-\dfunc)$
as
\begin{align}\label{eq:def:GAP ball}
	\GAPball(\pv, \dv) \triangleq \safeball(\GAPballc,\GAPballr)
\end{align}
where
\begin{subequations} \label{eq:def GAP ball} 
	\begin{align}
		\GAPballc \,\defeq\, & \dv
		\label{eq:def center GAP ball}                                       \\
		\GAPballr \,\defeq\, & \sqrt{\frac{2\gapfunc(\pv, \dv)}{\convCoeff}}
		\label{eq:def radius GAP ball}
	\end{align}
\end{subequations}
and \(\gapfunc(\pv,\dv) \defeq \pfunc(\pv) - \dfunc(\dv)\) is the so-called duality gap.
The popularity of the GAP ball is due to the two following assets:\vspace{0.2cm}
\begin{enumerate}
	\item The construction of the ball is valid for any problem satisfying our blanket hypotheses, that is \(\loss\), \(\reg\) are proper, closed, convex and  \(\loss\) is {\(\convCoeff^{-1}\)-Lipschitz} smooth over \(\kR^\obsdim\).
	\item A GAP ball can be constructed from any primal-dual feasible couple \((\pv, \dv)\).
	      In particular, under strong duality assumption and continuity of the duality gap over its domain, the radius of the ball can be made arbitrarily  small by choosing \((\pv, \dv)\) sufficiently close to some primal-dual solution \((\pvopt, \dvopt)\).\vspace{0.2cm} 
\end{enumerate}
These features have led the GAP ball to be widely applied and to allow for substantial acceleration performance in many setups, see \eg \cite{Dantas:2019sv,elvira2020safe,elvira2021safeSLOPE, guyard2022screen,Herzet:2019fj}.

Other constructions of safe balls, requiring either additional hypotheses on \(\loss\) and \(\reg\) or the knowledge of some specific  primal-dual couple \((\pv,\dv)\), have also been proposed in the literature.
All
the safe ball constructions (to the best of our knowledge) falling into the optimization framework considered in this paper are gathered in \Cref{table: safe regions} and will be reviewed in greater details in \Cref{sec:presentation holder ball}.
We note that, although requiring additional assumptions, some of these works put to the forth 
that the construction of safe balls smaller than the GAP ball is possible.
In particular, some contributions have highlighted two avenues for improvement.
In~\cite{Herzet16Screening}, the authors proposed a safe ball (referred to as ``FNE ball'') for the \LASSO problem and proved that it is a (potentially strict) \emph{subset} of the GAP ball.
More recently, the authors of~\cite{yamada2021dynamic} introduced the so-called ``dynamic EDPP ball'' and emphasized that the latter has \emph{a smaller radius} than the GAP ball constructed with the same primal-dual pair~\cite[\theoremInRef{10}]{yamada2021dynamic}.

In this paper, we provide a new mathematical framework gathering and extending these results to the general family of optimization problems~\eqref{eq:primal problem (fuzzy version)}.

\subsection{Contributions}

\input{table_balls}

The contribution of this paper is two-fold.
We first introduce a new safe-ball referred to as ``\FBI{} ball''. This ball is defined \(\forall (\pv,\dv)\in\domain(\pfunc)\times\domain(-\dfunc)\) as follows:\vspace{0.1cm} 
\begin{equation}
	\label{eq:theo: holder ball:ustar in holder ball}
	\HOLDERball(\pv, \dv)
	\defeq
	\safeball(\HOLDERballc, \HOLDERballr)
\end{equation}
\vspace{0.1cm}
where
\begin{subequations}
	\begin{align}
		\HOLDERballc \,\defeq\, & \frac{1}{2}\kparen{\dv  - \nabla \loss(\dicomat \pv)}
		\label{eq:def center FBI ball}                                                                                                \\
		\HOLDERballr \,\defeq\, & \sqrt{\frac{\gapfunc(\pv, \dv)}{\convCoeff} - \frac{\normTwo{\dv + \nabla\loss(\dicomat\pv)}^2}{4}}
		\label{eq:def radius FBI ball}
		,
	\end{align}
\end{subequations}
see \Cref{theo: holder ball}.\footnote{
	Note that the quantity under the square root in 
	\eqref{eq:def radius FBI ball} 
	is necessarily nonnegative (see~\Cref{lemma:u - r_x GAP inequality}).} 
The name ``\FBI'' stems from ``\textbf{R}efined Fenchel-\textbf{Y}oung ineq\textbf{u}ality'' and refers to the fact that the safeness of the ball is a consequence of the (double) application of a refined version of the well-known Fenchel-Young inequality (see \Cref{app:proofs FBI}).
Our ball construction is valid under
the same  generic assumptions as the GAP ball (see \Cref{sec:optimization framework} for a detailed discussion about our working hypotheses).
In particular:
\textit{i)}~it can be applied to any problem \eqref{eq:primal problem (fuzzy version)} involving a proper, closed, convex function \(\loss\) which is \(\convCoeff^{-1}\)-Lipschitz smooth over \(\kR^\obsdim\), and a proper, closed, convex function \(\reg\);
\textit{ii)}~primal-dual feasibility is the only assumption required for the pair \((\pv,\dv)\).

Second, we show that our safe ball construction generalizes or improves over all the existing results of the literature.
More specifically, we prove that all the existing safe balls correspond to particular cases or supersets of the proposed ball.
These results are summarized in the second column of \Cref{table: safe regions} and correspond to Propositions~\ref{prop:FBI ball subset GAP ball} to~\ref{prop: connection to SLORE and SFER balls} of the paper. 
As a byproduct, our analysis also provides a unified review of safe balls for problem~\eqref{def:dual solution} under hypotheses~\hypref{H2}-\hypref{H3} by connecting existing results in a common framework.

Interestingly, we note that the GAP ball is always a superset of the \Holder ball, that is:
\begin{equation} \label{eq:FBI ball subset GAP ball}
	\HOLDERball(\pv, \dv)\subseteq \GAPball(\pv, \dv),
\end{equation}
where the inclusion is shown to be strict as long as \((\pv, \dv)\) is not a primal-dual optimal couple, see \Cref{prop:FBI ball subset GAP ball}.
Moreover, a rapid inspection of \eqref{eq:def radius GAP ball} and \eqref{eq:def radius FBI ball} shows that the squared radius of the \Holder ball is never greater than half the squared radius of the GAP ball.

Since our construction is valid for any feasible\footnote{Strictly speaking, the proposed construction applies to any couple \((\pv,\dv)\in\kR^\pvdim\times\kR^\dvdim\) but (similarly to the GAP ball) it leads to a ball with infinite radius when \((\pv,\dv)\notin\domain(\pfunc)\times\domain(-\dfunc)\). This is the reason why we restrict our construction to \(\domain(\pfunc)\times\domain(-\dfunc)\) in the paper.} couple \((\pv, \dv)\) and holds under general assumptions on \(\loss\), \(\reg\), the results in \Cref{table: safe regions} therefore emphasize that the proposed framework unifies and generalizes all the methodologies previously proposed in the literature.

\subsection{Paper organization}
The rest of the paper is organized as follows.
In \Cref{sec:notational conventions}, we detail the notational conventions used in the paper.
In \Cref{sec:optimization framework}, we describe the working hypotheses considered in our derivations and discuss some of their implications.
\Cref{sec:presentation holder ball} is dedicated to the presentation of our new safe ball and its connection with the previous results of the literature.
Most of the technical details are deferred to \Cref{app:proofs comparison,app:proofs FBI,app:convex analysis}.
\vspace{0.2cm}

%% file: table_balls.tex
{
\def\yspacing{0.05cm}
\begin{table}[t]
	\centering
	\begin{tabular}{l l c c}
		\hline
		\textbf{Safe region} & \textbf{Relation}                                                   & \textbf{Cstr. on}  $(\pv,\dv)$ & \textbf{Hyp.  on $\loss$ and $\reg$} \\
		\hline
		GAP~\cite{fercoq2015mind, ndiaye2017gap,ndiaye2021}
		                     & $\supseteq \HOLDERball(\pv, \dv)$
		                     & \mbox{feasible}
		                     & \hypref{H2}-\hypref{H3}
		\\[\yspacing]

		\(\pv\)-GAP~\cite{Herzet2022}
		    & \(\supseteq \HOLDERball(\pv, \dv)\)
		    & \mbox{feasible}
		    & \hypref{H2}-\hypref{H3}, \(\reg=\lambda\|\cdot\|_1\) \\[\yspacing]

		Dyn. EDPP~\cite{yamada2021dynamic}
		                     & $= \HOLDERball(\dynEDPPballparam \pv, \dv)$
		                     & \mbox{feasible}
		                     & $\loss=\tfrac{1}{2}\|\obs-\cdot\|_2^2$,
		$\reg=$ gauge
		\\[\yspacing]

		FNE~\cite{Herzet16Screening}
		                     & $= \HOLDERball(\pv, \dv)$
		                     & $\ktranspose{\dicomat}\dv \in \partial\reg(\pv)$
		                     & $\loss=\tfrac{1}{2}\|\obs-\cdot\|_2^2$, $\reg=\lambda\|\cdot\|_1$
		\\[\yspacing]

		SASVI~\cite{pmlr-v32-liuc14}
		                     & $= \HOLDERball({\bf0}_\pvdim, \dv)$
		                     & $\dv\in\domain(-\dfunc)$
		                     & $\loss=\tfrac{1}{2}\|\obs-\cdot\|_2^2$, $\reg=\lambda\|\cdot\|_1$
		\\[\yspacing]

		EDPP~\cite{wang2015lasso}
		                     & $= \HOLDERball(\pv,\dv)$
		                     & \eqref{eq:def couple seq setup 1}-\eqref{eq:def couple seq setup 2}
		                     & $\loss=\tfrac{1}{2}\|\obs-\cdot\|_2^2$, $\reg=\lambda\|\cdot\|_1$
		\\[\yspacing]

		DPP~\cite{wang2015lasso}
		                     & $\supseteq \HOLDERball(\pv,\dv)$
		                     & \eqref{eq:def couple seq setup 1}-\eqref{eq:def couple seq setup 2}
		                     & $\loss=\tfrac{1}{2}\|\obs-\cdot\|_2^2$, $\reg=\lambda\|\cdot\|_1$
		\\[\yspacing]

		SAFE~\cite{Laurent-El-Ghaoui:2012qs}
		                     & $\supseteq \HOLDERball({\bf0}_\pvdim, \dv)$
		                     & $\dv\in\domain(-\dfunc)$
		                     & $\loss=\tfrac{1}{2}\|\obs-\cdot\|_2^2$, $\reg=\lambda\|\cdot\|_1$
		\\[\yspacing]

		SLORES~\cite{wang2014safe}
		                     & $\supseteq\HOLDERball(\pv,\dv)$
		                     & \eqref{eq:def couple seq setup 1}-\eqref{eq:def couple seq setup 2}
		                     & $\loss=\mbox{logistic}$, $\reg=\lambda\|\cdot\|_1$
		\\[\yspacing]

		SFER~\cite{pan2021safe}
		                     & $= \HOLDERball(\pv,\dv)$
		                     & \eqref{eq:def couple seq setup 1}-\eqref{eq:def couple seq setup 2}
		                     & $\loss=\mbox{logistic}$, $\reg=\lambda\|\cdot\|_1$
		\\[\yspacing]
		\hline
	\end{tabular}
	\vspace{0.5cm}
	\caption{
		Summary of the main safe-ball constructions proposed in the literature during the last decade. The first column provides the name and the references associated to the safe ball, the second describes its connection with the proposed \Holder ball, the third indicates the constraints on the primal-dual couple $(\pv,\dv)$ used in the construction.
		The last column specifies the setup considered by the authors in their work.
		\vspace{-0.5cm}
	}
	\label{table: safe regions}
\end{table}
}

%% file: manuscript_2_notations.tex

\section{Notations} \label{sec:notational conventions}
Unless mentioned explicitly, we will use the following notational conventions throughout the paper.
Vectors are denoted by lowercase bold letters (\textit{e.g.}, \(\pv, \aVariable\)) and matrices by uppercase bold letters (\textit{e.g.}, \(\dicomat\)).
We use the symbol ``\(\ktranspose{}\)'' to denote the transpose of a vector or a matrix.
The ``all-zero'' 
vector of dimension \(\pvdim\) is written \(\vecZero_{\pvdim}\)
.
\(\dotp{\aVariable}{\aVariable'}\) denotes the standard inner product between \(\aVariable\) and \(\aVariable'\). 
We use the notation \(\pvel_\idxatom\) to refer to the \(\idxatom\)th entry of a vector \(\pv\).
For matrices, we use \(\atom_\idxatom\) to denote the \(\idxatom\)th column of \(\dicomat\).
We let \(\kRbar = \kR\cup\{-\infty, +\infty\}\) where \(\kR\) refers to the set of real numbers. 
Given an extended real-valued function \(\kfuncdef{\aFunction}{\kR^\aDimension}{\kRbar}\), we let 
\begin{equation}
    \domain(\aFunction)\triangleq\kset{\aVariable\in\kR^\aDimension}{\aFunction(\aVariable)<+\infty}
    .
\end{equation}
%
The subdifferential of \(\kfuncdef{\aFunction}{\kR^\aDimension}{\kRbar}\) at \(\aVariable\in\kR^{\aDimension}\) is denoted \(\partial \aFunction(\aVariable)\).
It is defined for any \(\aVariable\in\domain(\aFunction)\) as 
\begin{equation}
    \partial \aFunction(\aVariable)
    \triangleq
    \kset{
	    \aSubgradient\in\kR^{\aDimension}
	}{
	    \forall \aVariable'\in\kR^{\aDimension}:\ \aFunction(\aVariable')\geq \aFunction(\aVariable)+\dotp{\aSubgradient}{\aVariable'-\aVariable}
	}.
\end{equation}
We refer to the elements of \(\partial \aFunction(\aVariable)\) as ``subgradients'', see~\cite[\definitionInRef{3.2}\pageInRef{35}]{Beck2017aa}.
Finally, 
\begin{equation}
\kfuncdef{\aFunction^*}{\kR^\aDimension}{\kRbar}[\aDualVariable][\sup_{\aVariable\in\kR^\aDimension} \
    \kangle{\aDualVariable,\aVariable} - \aFunction(\aVariable)],
\end{equation}
denotes the convex conjugate of \(\aFunction\), see~\cite[\definitionInRef{4.1}\pageInRef{87}]{Beck2017aa}. 
\vspace{0.3cm}


%% file: manuscript_3_framework.tex

\section{Optimization framework}\label{sec:optimization framework}
In this paper, we consider problem \eqref{eq:primal problem (fuzzy version)} with the following minimal assumptions:\vspace{0.3cm}
\begin{center}
	\begin{minipage}{.8\columnwidth}
		\iftoggle{arxiv}{
			\begin{enumerate}[label={(H\arabic*)}] 
				\item \(\loss\) and \(\reg\) are  proper, closed and convex functions. \label{H2}
				\item \(\loss\) is \(\convCoeff^{-1}\)-Lipschitz smooth over \(\kR^\obsdim\). \label{H3}
			\end{enumerate}
		}{
			\begin{enumerate}[(H1)] 
				\item \(\loss\) and \(\reg\) are  proper, closed and convex functions. \label{H2}
				\item \(\loss\) is \(\convCoeff^{-1}\)-Lipschitz smooth over \(\kR^\obsdim\). \label{H3}
			\end{enumerate}
		}
	\end{minipage}
\end{center}
\vspace{0.3cm}
%
%
We note that~\hypref{H2} and~\hypref{H3} correspond to the general hypotheses involved in the construction of the GAP ball, see \cite{ndiaye2021}.
In the rest of this section we elaborate on some properties of problems \eqref{eq:primal problem (fuzzy version)}-\eqref{def:dual solution} induced by these hypotheses.

First, since \(\loss\) (resp. \(\reg\)) is proper, closed and convex from~\hypref{H2}, its convex conjugate \(\loss^*\) (resp. \(\reg^*\)) is proper, closed and convex, see \cite[Theorems 4.5 and 4.13]{Beck2017aa}. 
Moreover, the convexity and \(\convCoeff^{-1}\)-Lipschitz smoothness over \(\kR^\obsdim\) of \(\loss\) in~\hypref{H3} imply that \(\loss^*\) is \(\convCoeff\)-strongly convex, see \cite[\theoremInRef{5.26}]{Beck2017aa}, that is:
\begin{equation}
	\mbox{\(\forall \aVariable,\aVariable'\in\domain(\loss^*)\) and \(\bfg\in\partial\loss^*(\aVariable)\)}:\
	\loss^*(\aVariable')
	\geq
	\loss^*(\aVariable)
	+ \dotp{\bfg}{\aVariable'-\aVariable}
	+ \tfrac{\convCoeff}{2} \|\aVariable'-\aVariable\|_2^2
	.
\end{equation}
%
Under the properness assumption in \hypref{H2}, the duality gap, defined as
\begin{equation}
	\kfuncdef{\gapfunc}{\kR^n\times\kR^m}{\kRbar}[(\pv,\dv)][\pfunc(\pv) - \dfunc(\dv)]
\end{equation}
is always a nonnegative quantity, see ~\cite[Item~(i) of \propositionInRef{15.21}\pageInRef{254}]{Bauschke2017}. Moreover, \(\gapfunc(\pv,\dv)<+\infty\) if and only if \((\pv,\dv)\in\domain(\pfunc)\times\domain(-\dfunc)\).
Hypotheses \hypref{H2}-\hypref{H3} also imply that strong duality holds for some primal-dual couple as emphasized by the following lemma:
\vspace{0.1cm}

\begin{lemm} \label{lemma:strong duality}
	Let \(\pvopt\) be a minimizer of \eqref{eq:primal problem (fuzzy version)}.
	If \hypref{H2}-\hypref{H3} hold, then there exists \(\dvopt\in\kR^\dvdim\) such that
	\begin{equation}\label{eq:primal-dual solution}
		\gapfunc(\pvopt,\dvopt)=0
		.
	\end{equation}
\end{lemm}
\begin{proof}
	If \hypref{H2} is verified, we have from \cite[Theorems~15.23 and 15.24.(viii)]{Bauschke2017}
	that strong duality holds and a maximizer \(\dvopt\) to \eqref{def:dual solution} exists provided that 
	\begin{equation} \label{eq:sufficient condition for strong duality}
		\mathrm{relint}(\domain(\loss))\cap \mathrm{relint}(\dicomat\domain(\reg))\neq \emptyset
		,
	\end{equation}
	where \(\mathrm{relint}(\cdot)\) denotes the relative interior of a set.
	Now, under our \(\alpha^{-1}\)-Lipschitz smoothness assumption~\hypref{H3}, we have that \(\domain(\loss)=\kR^\dvdim\) so that condition~\eqref{eq:sufficient condition for strong duality} reduces to \(\mathrm{relint}({\dicomat}\domain(\reg))\neq \emptyset\).
	Since \(\reg\) is a proper convex function, its domain \(\domain(\reg)\) is non-empty (by definition of properness) and convex~\cite[Section~2.3.1]{Beck2017aa}.
	The set \(\dicomat\domain(\reg)\) is thus also non-empty convex as the image of a non-empty convex set under a linear operator \(\dicomat\)~\cite[\propositionInRef{3.5}]{Bauschke2017}.
	Therefore, the relative interior of \(\dicomat\domain(\reg)\) is non-empty as a consequence of~\cite[\theoremInRef{3.17}]{Beck2017aa}.
\end{proof}
We note that any \(\dvopt\) verifying \eqref{eq:primal-dual solution} must obviously be a solution of \eqref{def:dual solution}.
Hence, we have from \Cref{lemma:strong duality} that a maximizer of \eqref{def:dual solution} exists under \hypref{H2}-\hypref{H3}.
Moreover, the \(\convCoeff\)-strong convexity of \(\loss^*\) implies that this  maximizer is unique, see \cite[\theoremInRef{5.25}]{Beck2017aa}.
Finally, since strong duality holds,
the following conditions must be satisfied by any primal-dual optimal couple \((\pvopt,\dvopt)\), see \cite[\theoremInRef{19.1}\pageInRef{330}]{Bauschke2017}:
\begin{align} \label{eq:optimality condition ustar = -nabla f(Axstar)}
	\dvopt                      & = - \nabla \loss(\dicomat \pvopt)                                           \\
	\ktranspose{\dicomat}\dvopt & \in \partial \reg(\pvopt) \label{eq: optimality condition: u dual feasible}
	.
\end{align}

%% file: manuscript_4_FBI.tex
\section{The \FBI{} framework and its connection to the state of the art} \label{sec:presentation holder ball}

The main theoretical result of this paper is the following new safe ball:
\begin{theo}[\Holder ball] \label{theo: holder ball}
	Assume \hypref{H2}-\hypref{H3} hold true.
	Then, we have for any \((\pv, \dv) \in
	\domain(\pfunc)\times\domain(-\dfunc)\):
	\begin{equation}
		\dvopt \in \HOLDERball(\pv, \dv)
		\defeq
		\safeball(\HOLDERballc, \HOLDERballr)
	\end{equation}
	where
	\begin{subequations}
		\begin{align}
			\HOLDERballc \,\defeq\, & \frac{1}{2}\kparen{\dv  - \nabla \loss(\dicomat \pv)}
			\\
			\HOLDERballr \,\defeq\, & \sqrt{\frac{\gapfunc(\pv, \dv)}{\convCoeff} - \frac{\normTwo{\dv + \nabla\loss(\dicomat\pv)}^2}{4}}
			.
		\end{align}
	\end{subequations}
\end{theo}
The name
``\FBI{}'' stands for ``\textbf{R}efined Fenchel-\textbf{Y}oung ineq\textbf{u}ality'', a central element appearing in the proof of the safeness of the proposed region, detailed in \Cref{subsec:proof gap inequality}.


A close inspection of the hypotheses of \Cref{theo: holder ball} reveals that the construction of the \FBI{} ball is applicable under exactly the same assumptions as the GAP ball, that is: \textit{i)} it holds for any problem satisfying hypotheses \hypref{H2}-\hypref{H3} on \(\loss\) and \(\reg\); \textit{ii)} it is valid for any primal-dual feasible couple \((\pv,\dv)\).
Despite of its generality, a careful examination of the definition of the radius of the GAP and \FBI{} balls in~\eqref{eq:def radius GAP ball} and~\eqref{eq:def radius FBI ball}
indicates that -- given a feasible primal-dual pair \((\pv,\dv)\) -- the squared radius of the \Holder ball is \emph{always} at least twice as small as the squared radius of the GAP ball.
In fact, as emphasized in \Cref{sec:GAP ball} below, the GAP ball is a strict superset of the proposed \Holder ball for any feasible primal-dual \((\pv,\dv)\) different from $(\pvopt,\dvopt)$.

%% file: manuscript_5_comparison.tex


In the rest of this section, we explain how the safe balls previously proposed in the literature relate to the \Holder ball.
More specifically, we emphasize that the previous results of the state of the art can be seen as either particular cases or supersets of the proposed ball.
Our results are contained in Propositions \ref{prop:FBI ball subset GAP ball} to \ref{prop: connection to SLORE and SFER balls} and are summarized in the second column of \Cref{table: safe regions}.
The third and fourth columns of the table specify assumptions necessary for constructing the corresponding safe ball: the third column details the nature of the functions \(\loss\), \(\reg\) defining \eqref{eq:primal problem (fuzzy version)}, while the fourth column outlines potential constraints on the primal-dual pair \((\pv, \dv)\) used in the construction of the ball.

Before proceedings to the connection between the \Holder{} ball  and the existing results of the literature, let us make two important remarks regarding the choice the couple \((\pv,\dv)\) involved in the construction of the safe ball.

First, an approach which have been considered (often implicitly) in many contributions of the literature consists in choosing a feasible pair \((\pv,\dv)\) such that
\begin{equation}\label{eq:condition couple u=-nabla f(Ax)}
	\ktranspose{\dicomat}\dv \in \partial\reg(\pv)
	.
\end{equation}
Interestingly, when \eqref{eq:condition couple u=-nabla f(Ax)} is satisfied the function \(\gapfunc\) can be related to two other well-known quantities, namely the Fenchel divergence of \(\loss\) (see \cite[\definitionInRef{2}\pageInRef{9}]{Blondel2020}) and the Bregman divergence of \(\loss^*\)
In particular, the following lemma holds:
\begin{lemm}\label{lemma: connection between gap and divergence}
	Assume \hypref{H2}-\hypref{H3} hold and let\footnote{
		\(\Fen(\pv,\dv)\) corresponds to the Fenchel divergence of \(\loss\) evaluated at \((\dicomat\pv,-\dv)\); \(\Breg(\pv,\dv)\) is reminiscent from the Bregman divergence of \(\loss^*\) evaluated at \((-\dv,\nabla\loss(\dicomat\pv))\). 
			We acknowledge that, in contrast to the standard definition of the Bregman divergence,~\eqref{def:Bregman divergence} does not impose differentiability of \(\loss^*\). 
			However, under~\hypref{H3}, the differentiability of \(\loss\) implies that \(\partial \loss(\dicomat\pv)=\{\nabla \loss(\dicomat\pv)\}\), and, since \(\loss\) is also convex, proper, and closed under~\hypref{H2} (see \Cref{lemma:subdifferential inversion}), we have \(\dicomat\pv\in \partial f^*(\nabla f(\dicomat\pv))\).
			Furthermore, when \(\loss^*\) is differentiable, we recover \(\dicomat\pv=\nabla\loss^*(\nabla\loss(\dicomat\pv))\), and~\eqref{def:Bregman divergence}  coincides with the standard definition of the Bregman divergence.
	}
	\begin{align}
		{\Fen}(\pv,\dv) & \triangleq \loss(\dicomat\pv) + \loss^*(-\dv) + \dotp{\dv}{\dicomat\pv}
		\label{def:Fenchel divergence}                                                                                                  \\
		\Breg(\pv,\dv)  & \triangleq \loss^*(-\dv) - \loss^*(\nabla\loss(\dicomat\pv))+\dotp{\dicomat\pv}{\dv+\nabla\loss(\dicomat\pv)}
		\label{def:Bregman divergence}
		.
	\end{align}
	If \eqref{eq:condition couple u=-nabla f(Ax)} is verified, then \(\forall(\pv,\dv)\in\domain(\pfunc)\times\domain(-\dfunc)\):
	\begin{equation}\label{eq:connection between gap and divergence}
		\gapfunc(\pv,\dv)=\Fen(\pv,\dv)=\Breg(\pv,\dv).
	\end{equation}
\end{lemm}
We refer the reader to \Cref{proof:lemma: connection between gap and divergence} for a proof of this result.
The connection between the duality gap and the Fenchel/Bregman divergences is of interest in two respects.
On the one hand, these divergences are sometimes more straightforward to express than the duality gap and thus give an alternative formulation to the proposed \FBI{} ball under the particular assumption \eqref{eq:condition couple u=-nabla f(Ax)}.
On the other hand, some contributions of the literature (see \Cref{sec:SLORE and SFER balls}) have directly expressed their safe ball as a function of \(\Breg(\pv,\dv)\).
The connection established in \Cref{lemma: connection between gap and divergence} will thus allow us to make a direct link between these works and the \FBI{} framework proposed in this paper.

Second, we mention that the following definition of primal-dual couple \((\pv,\dv)\) has been considered in many contributions of the literature (see \eg{} \cite{ pmlr-v32-liuc14, pan2021safe,wang2015lasso, wang2014safe}):
\begin{equation}\label{eq:def couple seq setup 1}
	(\pv,\dv)\triangleq \left(\pvopt_{\gamma},-{\gamma}^{-1}\nabla \loss(\dicomat\pvopt_{\gamma})\right)
\end{equation}
where \(\gamma>0\) and 
\begin{equation} \label{eq:def couple seq setup 2}
	\pvopt_{\gamma} \in \kargmin_{\pv\in\kR^\pvdim} \loss(\dicomat\pv)+ \gamma \reg({\pv}).
\end{equation}
This type of construction appears for example in ``sequential'' settings where one wants to solve \eqref{eq:primal problem (fuzzy version)} for \(\reg(\cdot) = \lambda \|\cdot\|\) where \(\|\cdot\|\) denotes some norm on \(\kR^\pvdim\) and the solution of a similar problem with \(\reg(\cdot) = \lambda_0 \|\cdot\|\) has already been computed previously.
The solution of the latter problem can then be expressed as in \eqref{eq:def couple seq setup 2} with \(\gamma=\tfrac{\lambda_0}{\lambda}\) and \(\reg(\cdot) = \lambda \|\cdot\|\).

The next lemma emphasizes that \eqref{eq:def couple seq setup 1}-\eqref{eq:def couple seq setup 2} correspond in fact to a particular strategy to build primal-dual couples obeying \eqref{eq:condition couple u=-nabla f(Ax)}:  \vspace{0.2cm}
\begin{lemm}\label{lemma: generating primal-dual couple with rescaled problem}
	If \((\pv,\dv)\) is defined as in \eqref{eq:def couple seq setup 1}-\eqref{eq:def couple seq setup 2}, then it verifies \eqref{eq:condition couple u=-nabla f(Ax)}.
	\vspace{0.2cm}
\end{lemm}
{
	\begin{proof}
		Since \(\gamma\) is assumed positive,~\eqref{eq:def couple seq setup 2} can be equivalently rewritten:
		\begin{equation} \label{eq:def couple seq setup 2 rewritting}
			\pvopt_{\gamma} \in \kargmin_{\pv\in\kR^{\pvdim}} \gamma^{-1}\loss(\dicomat\pv)+ \reg({\pv})
			.
		\end{equation}
		The functions \(\gamma^{-1}\loss\) and \(\reg\) are convex, proper and closed from~\hypref{H2} and \(\gamma^{-1}\loss\) is \((\gamma\convCoeff)^{-1}\)-Lipschitz smooth over \(\kR^{\dvdim}\).
		One obtains~\eqref{eq:condition couple u=-nabla f(Ax)} by expanding the optimality conditions~\eqref{eq:optimality condition ustar = -nabla f(Axstar)}-\eqref{eq: optimality condition: u dual feasible} associated to~\eqref{eq:def couple seq setup 2 rewritting} and its dual.
	\end{proof}
}

Since many safe balls proposed in the literature rely on the particular construction \eqref{eq:def couple seq setup 1}-\eqref{eq:def couple seq setup 2}, the result in \Cref{lemma: generating primal-dual couple with rescaled problem}
emphasizes that these constructions in fact consider a primal-dual feasible pair \((\pv,\dv)\) verifying \eqref{eq:condition couple u=-nabla f(Ax)} and that (from \Cref{lemma: connection between gap and divergence}) the connection \eqref{eq:connection between gap and divergence} between the duality gap and the Fenchel/Bregman divergences is thus in force.

\vspace{0.2cm}
\subsection{GAP balls}
\label{sec:GAP ball}
The GAP ball first proposed in \cite{fercoq2015mind} and later generalized in \cite{ndiaye2017gap,ndiaye2021} is defined in \eqref{eq:def center GAP ball}-\eqref{eq:def radius GAP ball}.
Its construction is valid under assumptions \hypref{H2}-\hypref{H3} and can take any primal-dual \((\pv,\dv)\) as input.
The next result shows that the \Holder ball is always a subset of the GAP ball: \vspace{0.2cm}

\begin{prop}[The GAP ball contains the \Holder ball] \label{prop:FBI ball subset GAP ball}
	Assume \hypref{H2}-\hypref{H3} hold. Then, for any primal-dual pair \((\pv, \dv) \in
	\domain(\pfunc)\times\domain(-\dfunc)\):
	\begin{equation} \label{eq:FBI ball subset GAP ball B}
		\HOLDERball(\pv, \dv)\subseteq
		\safeball(\GAPballc,\GAPballr)
		.
	\end{equation}
	Moreover, the inclusion is strict as soon as the primal-dual pair \((\pv, \dv)\) is
	not optimal. 
\end{prop}
In \cite[Section IV.B]{Herzet2022} a variant of the GAP ball for the specific case where \(\reg = \lambda\|\cdot\|_1\) was proposed.
Although the construction of the ball presented in~\cite{Herzet2022} holds in a slightly more general setup,\footnote{In particular, a slightly weaker version of~\hypref{H3} is considered.} we focus hereafter on the case where \(\loss\) is proper, closed, convex and satisfies~\hypref{H3}.
The center and radius of this ball (referred to as \(\pv\)-GAP since its center depends on \(\pv\) instead of \(\dv\) as in the standard GAP ball) reads as follows:
\begin{subequations} \label{eq:xpball:defs}
	\begin{align}
		\xGAPballc & \triangleq -\nabla \loss(\dicomat\pv)                    \\
		\xGAPballr & \triangleq \sqrt{\frac{2\gapfunc(\pv, \dv)}{\convCoeff}} \label{eq:xpball:def radius}
	\end{align}
\end{subequations}
where \((\pv,\dv)\) can be any primal-dual feasible couple.
Similarly to \Cref{prop:FBI ball subset GAP ball}, the next result shows that the \(\pv\)-GAP ball is also a superset of the proposed \FBI{} region: 

\begin{prop}[The \(\pv\)-GAP ball contains the \Holder ball] \label{prop:FBI ball subset x-GAP ball}
	Assume \hypref{H2}-\hypref{H3} hold. Then, for any primal-dual pair \((\pv, \dv) \in
	\domain(\pfunc)\times\domain(-\dfunc)\):
	\begin{equation} \label{eq:FBI ball subset xGAP ball}
		\HOLDERball(\pv, \dv)\subseteq
		\safeball(\xGAPballc,\xGAPballr)
		.
	\end{equation}
	Moreover, the inclusion is strict as soon as the primal-dual pair \((\pv, \dv)\) is
	not optimal. 
\end{prop}
The proof of \Cref{prop:FBI ball subset GAP ball,prop:FBI ball subset x-GAP ball} is given in \Cref{proof:FBI subset GAP}.
\vspace{0.2cm}


\subsection{Dynamic EDDP ball}

In~\cite{yamada2021dynamic}, the authors focused on the particular family of problems where
\begin{align}
	\loss & = \tfrac{1}{2}\|\obs - \cdot\|_2^2\label{eq:dynEDDP:H1} \\
	\reg  & = \lambda\|\cdot\|\label{eq:dynEDDP:H2}
\end{align}
for some vector \(\obs\in\kR^\dvdim\), scalar \(\lambda>0\) and norm \(\|\cdot\|\).\footnote{{More precisely, the authors of~\cite{yamada2021dynamic} considered gauge functions instead of norms for the definition of \(\reg\). Although the results presented in this section still hold in this more general setup, we stick to norms to simplify the exposition.}}
They introduced a new safe ball (see~\cite[\theoremInRef{9}]{yamada2021dynamic}), dubbed ``dynamic EDPP ball'' and valid for any primal-dual feasible couple \((\pv,\dv)\).
The center and radius  of the dynamic EDPP ball reads as follows:
\begin{subequations}
	\begin{align}
		\dynEDPPballc \,=\, & \frac{1}{2}(\obs+\dv - \dynEDPPballparam \dicomat \pv)
		\label{eq:def dynamic EDPP ball:center}                                                                      \\
		\dynEDPPballr \,=\, & \frac{1}{2} \sqrt{\kvvbar{\obs-\dv}_2^2 -  \Vert{\displaystyle\dynEDPPballparam}\dicomat \pv\Vert_2^2}
		\label{eq:def dynamic EDPP ball:radius}
	\end{align}
\end{subequations}
where \(\dynEDPPballparam\) is defined as (with the conventions \(0/0=0\) and \(1/0=+\infty\)):
\begin{equation} \label{eq:EDPP:def definition gamma0}
	\dynEDPPballparam =
	\max \left( 0, \frac{\dotp{\dicomat\pv}{\obs+\dv}
		- 2\lambda\norm{\pv}}{\kvvbar{\dicomat\pv}_2^2} \right)
	.
\end{equation}
The connection between the \Holder and dynamic EDPP balls is established in the following proposition:
\vspace{0.2cm}
\begin{prop}[Dynamic EDPP ball is a special case of the \Holder{} ball]
	\label{cor: dynamic EDPP ball}
	Assume \eqref{eq:dynEDDP:H1}-\eqref{eq:dynEDDP:H2} holds.
	Then,  for any \((\pv, \dv)\in \domain(\pfunc)\times\domain(-\dfunc)\): \vspace{0.2cm}
	\iftoggle{arxiv}{
		\begin{enumerate}[label=\roman*)]
	}{
		\begin{enumerate}[i)]
	}
		\item The quantity \(\dynEDPPballparam\) defined in~\eqref{eq:EDPP:def definition gamma0} verifies
		      \begin{equation*} 
			      \dynEDPPballparam
			      \in
			      \kargmin_{\aScaling\geq0}\
			      \sqrt{
				      \gapfunc(\aScaling\pv, \dv) - \frac{1}{4}\normTwo{\dv + \nabla\loss(\aScaling\dicomat\pv)}^2
			      }
			      .
		      \end{equation*}
		      \label{item:dynamic EDPP ball:definition gamma}
		\item We have
		      \begin{equation*}
			      \safeball(\dynEDPPballc, \dynEDPPballr)= \HOLDERball(\dynEDPPballparam\pv, \dv)
			      .
		      \end{equation*}
		      \label{item:dynamic EDPP ball:special case holder ball}
	\end{enumerate}
\end{prop}
A proof of this result is available in \Cref{subsec:proof connection EDPP ball}.
In particular, the dynamic EDPP ball was shown to exhibit a smaller radius as compared to the GAP ball constructed with the (feasible) primal-dual pair \((\pv,\dv)\), see~\cite[\theoremInRef{10}]{yamada2021dynamic}.
Item~\ref{item:dynamic EDPP ball:special case holder ball} of \Cref{cor: dynamic EDPP ball} (combined with \Cref{prop:FBI ball subset GAP ball}) elucidates this connection by showing that the dynamic EDPP ball is in fact a subset of the GAP ball \(\GAPball(\dynEDPPballparam\pv, \dv)\).
Finally,
item~\ref{item:dynamic EDPP ball:definition gamma} of \Cref{cor: dynamic EDPP ball} provides a novel interpretation of the definition of \(\dynEDPPballparam\), that is \(\dynEDPPballparam\) corresponds to a nonnegative rescaling of the primal vector \(\pv\) minimizing the radius of the \Holder ball.


\vspace{0.2cm}
\subsection{FNE, EDDP, DPP and SASVI balls}

FNE \cite{Herzet16Screening}, (E)DDP \cite{wang2015lasso} and SASVI \cite{pmlr-v32-liuc14} balls are safe regions designed for the same problem where
\begin{align}
	\loss & = \tfrac{1}{2}\|\obs - \cdot\|_2^2\label{eq:FNE:H1} \\
	\reg  & = \lambda\|\cdot\|_1\label{eq:FNE:H2}
	.
\end{align}
They all assume (explicitly or implicitly) that the primal-dual couple \((\pv,\dv)\) used in the construction verifies \eqref{eq:condition couple u=-nabla f(Ax)}.
We start with the description of the FNE ball which corresponds to the most general construction. We address the EDDP, DPP and SASVI balls at the end of the section as particular cases or relaxation of the FNE region.

The FNE ball is defined by the following center and radius:
\begin{subequations}
	\begin{align}
		\FNEballc \,=\, & \dv + \frac{1}{2}(\obs- \dicomat \pv - \dv)
		\label{eq:def FNE ball:center}                                   \\
		\FNEballr \,=\, & \frac{1}{2}\normTwo{\obs - \dicomat \pv - \dv}
		.
		\label{eq:def FNE ball:radius}
	\end{align}
\end{subequations}
In \cite[Theorem 1]{Herzet16Screening}, the authors showed that the FNE ball is safe for any primal-dual feasible pair \((\pv, \dv)\) satisfying
\begin{equation} \label{eq: hat opt 2}
	\dotp{\dv}{\dicomat \pv} = \lambda\normOne{\pv}
	.
\end{equation}
The following result shows that the FNE ball in fact corresponds to a particular case of the \FBI{} ball when \eqref{eq:FNE:H1}-\eqref{eq:FNE:H2} and \eqref{eq: hat opt 2} hold:\vspace{0.2cm}

\begin{prop}[FNE ball is a special case of the \Holder ball] \label{prop:Sequential FNE ball is a special case of Holder ball}
	Assume \eqref{eq:FNE:H1}-\eqref{eq:FNE:H2} holds.
	Then, for any \((\pv, \dv)\in \domain(\pfunc)\times\domain(-\dfunc)\) satisfying~\eqref{eq: hat opt 2}:
	\begin{equation}
		\safeball(\FNEballc, \FNEballr)
		=
		\HOLDERball(\pv, \dv)
		.
	\end{equation}
\end{prop}
\vspace{0.2cm}
A proof of this result is available in \Cref{subsec:proof connection FNE ball}.
We note that the authors also showed in \cite[\lemmaInRef{1}]{Herzet16Screening} that the FNE ball is a strict subset of the GAP ball as long as \((\pv,\dv)\neq(\pvopt,\dvopt)\). Interestingly, in view of \Cref{prop:Sequential FNE ball is a special case of Holder ball}, this result turns out to be a particular case of  \Cref{prop:FBI ball subset GAP ball} in the more general framework of the \FBI{} ball.

\vspace*{.2cm}

The EDPP and SASVI balls represent specific instances of the FNE ball, resulting from particular choices of the pair \((\pv, \dv)\).
On the one hand, the SASVI ball (see \cite[Section 2.2]{pmlr-v32-liuc14}) corresponds to the case where \((\pv,\dv)=({\bf0}_{\pvdim},\dv)\) for some dual feasible point \(\dv\), \ie
\begin{subequations}
	\begin{align}
		\SASVIballc \,=\, & \frac{1}{2}(\obs+ \dv)
		\label{eq:def SASVI ball:center}                    \\
		\SASVIballr \,=\, & \frac{1}{2}\normTwo{\obs - \dv}
		\label{eq:def SASVI ball:radius}
		.
	\end{align}
\end{subequations}
It is easy to see this couple trivially verifies \eqref{eq: hat opt 2}.
Using \Cref{prop:Sequential FNE ball is a special case of Holder ball} with \((\pv,\dv)=({\bf0}_{\pvdim},\dv)\), \(\dv\in\domain(-\dfunc)\), then directly leads to
\begin{equation}
	\nonumber
	\safeball(\SASVIballc,\SASVIballr)=\HOLDERball({\bf0}_{\pvdim},\dv)
	.
\end{equation}
On the other hand, the center and radius of the EDDP ball (see \cite[Theorem 13]{wang2015lasso}) obeys the same definition \eqref{eq:def FNE ball:center}-\eqref{eq:def FNE ball:radius} as those of the FNE ball but
for \((\pv,\dv)\) defined as in \eqref{eq:def couple seq setup 1}-\eqref{eq:def couple seq setup 2} for some \(\gamma>0\).
Taking into account that
\begin{equation}\label{eq:subdiff norm l1}
	\partial\|\pv\|_1
	= \kset{\aVariable\in\kR^\pvdim}{\dotp{\aVariable}{\pv}=\|\pv\|_1, \|\aVariable\|_\infty\leq 1},
\end{equation}
it is easy to see that \eqref{eq: hat opt 2} together with feasibility of \((\pv,\dv)\) is in fact an equivalent rewriting of \eqref{eq:condition couple u=-nabla f(Ax)} for \(\reg = \lambda\|\cdot\|_1\).
Hence, in view of \Cref{lemma: generating primal-dual couple with rescaled problem}, the couple \((\pv,\dv)\) considered in the EDDP construction verifies \eqref{eq: hat opt 2} and this ball is nothing but a particular instance of FNE ball.
\Cref{prop:Sequential FNE ball is a special case of Holder ball} thus applies for the EDDP ball as well.

Finally, it was shown in \cite[Theorem 13]{wang2015lasso} that the DPP ball is always a superset of the EDPP ball.
This directly leads to the inclusion reported in \Cref{table: safe regions}.
\vspace{0.2cm}

\subsection{SAFE ball}
The SAFE ball is the first safe region proposed in the seminal paper \cite{Laurent-El-Ghaoui:2012qs}.
It applies in the case where
\begin{align}
	\loss & = \tfrac{1}{2}\|\obs - \cdot\|_2^2\label{eq:SAFE:H1} \\
	\reg  & = \lambda\|\cdot\|_1\label{eq:SAFE:H2}
	.
\end{align}
Its center and radius are defined \(\forall \dv\in\domain(-\dfunc)\) as
\begin{align}
	\SAFEballc & = \obs            \\
	\SAFEballr & = \|\obs-\dv\|_2.
\end{align}
In \Cref{subsec:proof:prop: connection to SAFE ball}, we show that the SAFE ball is a relaxation of the proposed \Holder{} ball for \((\pv,\dv)=({\bf0}_\pvdim,\dv)\) with
\(\dv\in\domain(-\dfunc)\).
More specifically, we prove that the following result holds: \vspace{0.2cm}
\begin{prop}\label{prop: connection to SAFE ball}
	Assume \eqref{eq:SAFE:H1}-\eqref{eq:SAFE:H2} holds.
	Then, for any
	\(\dv\in\domain(-\dfunc)\):
	\begin{equation}\label{prop:eq:inclusion FBI in SAFE}
		\safeball(\SAFEballc, \SAFEballr)
		\supseteq
		\HOLDERball({\bf0}_\pvdim, \dv)
		.
	\end{equation}
\end{prop}
\vspace{0.2cm}

\subsection{SLORE and SFER balls}\label{sec:SLORE and SFER balls}

We end up this section by considering the SLORES and SFER balls respectively proposed in~\cite[\theoremInRef{2}]{wang2014safe} and~\cite[\corollaryInRef{1}]{pan2021safe}.
The focus of these papers is on problem \eqref{eq:primal problem (fuzzy version)} with the following definitions for \(\loss\) and \(\reg\): 
\begin{align}
	\loss(\aVariable) & = \sum_{\idxobs=1}^\dvdim\log(1 + \cste^{-\aScalar_\idxobs}),\label{eq:SLORE:H1} \\
	\reg(\aVariable)  & = \lambda\|\aVariable\|_1\label{eq:SLORE:H2}
	.
\end{align}
The construction of these balls is moreover based on the knowledge of a primal-dual couple \((\pv,\dv)\) verifying \eqref{eq:def couple seq setup 1}-\eqref{eq:def couple seq setup 2} for some \(\gamma>0\). 
The expression of the center and radius of the SLORE and SFER balls respectively read as
\begin{subequations}
	\begin{align}
		\seqSLOREballc & \,=\, \gamma \dv
		\label{eq:def seq SLORE ball:center} \\
		\seqSLOREballr & \,=\, \sqrt{
			\frac{1}{2}
			\Breg(\pv,\dv)
		}
		.
		\label{eq:def seq SLORE ball:radius}
	\end{align}
\end{subequations}
and
\begin{subequations}
	\begin{align}
		\seqFERballc & \,=\, \frac{1}{2}(1+\gamma)\dv
		\label{eq:def seq FER ball:center}            \\
		\seqFERballr & \,=\, \sqrt{
			{\frac{1}{4}}
			\Breg(\pv,\dv)
			- \frac{1}{4}\normTwo{(1-\gamma)\dv}^2
		}
		.
		\label{eq:def seq FER ball:radius}
	\end{align}
\end{subequations}
In \cite[\theoremInRef{3}]{pan2021safe}, it was shown that 
\begin{equation}
		\safeball(\seqFERballc,\seqFERballr)
		\subseteq
		\safeball(\seqSLOREballc, \seqSLOREballr).
\end{equation}
The next result shows that the SFER ball is a particular instance of RYU ball, thereby proving the results in \Cref{table: safe regions}:
\vspace*{.2cm}
\begin{prop}\label{prop: connection to SLORE and SFER balls}
	Assume \eqref{eq:SLORE:H1}-\eqref{eq:SLORE:H2} hold and \((\pv,\dv)\) is defined as \eqref{eq:def couple seq setup 1}-\eqref{eq:def couple seq setup 2} for some \(\gamma>0\). Then, we have
	\begin{equation}\label{eq:connection FBI, SLORE, SFER}
		\safeball(\seqFERballc,\seqFERballr)
		=
		\HOLDERball(\pv, \dv)
		.
	\end{equation}
\end{prop}
A proof of this result is available in \Cref{subsec:proof around FBI inequality for logistic loss + ell_1}.

%% file: conclusion.tex
\vspace{0.2cm}
\section{Conclusion}

In this paper, we introduced a new framework for constructing safe balls for a broad class of optimization problems. 
Specifically, our approach addresses cases where the cost function consists of a closed, proper, convex, Lipschitz-smooth term combined with another closed, proper, convex term, and only relies on the knowledge of a primal-dual feasible pair. 
The proposed construction not only unifies existing methods but also extends and improves upon all prior approaches from the last decade, providing a comprehensive framework for generating safe balls within this family of optimization problems and connecting existing results of the literature.

\vspace{0.2cm}

\iftoggle{pagebreakSection}{
	\vfill
	\pagebreak	
}{
}

%% file: manuscript_A_convex_considerations.tex

\section{Convex analysis} \label{app:convex analysis}
This appendix reminds two standard results from convex analysis.
The first result relates the subdifferential of a function to the subdifferential of its convex conjugate:

\begin{lemm}[Subdifferential inversion] \label{lemma:subdifferential inversion}
	Let \(\kfuncdef{\aFunction}{\kR^\aDimension}{\kRbar}\) be a proper, closed and convex function.
	Then, for all \(\aVariable,\aDualVariable\in\kR^\aDimension\): 
	\begin{equation} \label{eq:subdifferential inversion}
		\aDualVariable\in\partial\aFunction(\aVariable)
		\;\Longleftrightarrow\;
		\aVariable\in\partial\aFunction^*(\aDualVariable)
		.
	\end{equation}
\end{lemm}
A proof of this lemma can be found in \cite[\theoremInRef{4.20}\pageInRef{105}]{Beck2017aa}.

The second result recalls two Fenchel-Young inequalities:

\begin{lemm}[Fenchel-Young inequalities] \label{theo: FY inequalities}
	Let \(\kfuncdef{\aFunction}{\kR^\aDimension}{\kRbar}\) be a proper and convex function. Then, 
	for all \(\aVariable,\aDualVariable\in\kR^\aDimension\):
	\begin{equation} \label{theo: FY inequalities 1}
		\aFunction(\aVariable) + \aFunction^*(\aDualVariable)
		\geq
		\dotp{\aDualVariable}{\aVariable}
	\end{equation}
	with equality if and only if \(\aDualVariable \in \partial \aFunction(\aVariable)\). 

	If \(\aFunction\) is moreover closed and \(\convCoeff\)-strongly convex, then for all \(\aVariable,\aDualVariable\in\kR^\aDimension\):
	\begin{equation}
		\label{theo: FY inequalities 2}
		\aFunction(\aVariable) + \aFunction^*(\aDualVariable)
		\geq
		\dotp{\aDualVariable}{\aVariable}
		+ \frac{\convCoeff}{2}\norm{\aVariable - \nabla \aFunction^*(\aDualVariable)}_2^2.
	\end{equation}
\end{lemm}

We note that~\eqref{theo: FY inequalities 1} corresponds to the standard formulation of the well-known Fenchel-Young inequality. 
A proof of this result follows from \cite[\theoremInRef{4.6}\pageInRef{88}]{Beck2017aa} and \cite[\theoremInRef{4.20}\pageInRef{104}]{Beck2017aa}.
\eqref{theo: FY inequalities 2} is a refined version of the Fenchel-Young inequality which applies to closed and strongly convex functions.
Since the latter is less common in the literature, a proof is provided hereafter.\footnote{
	Another version of the proof can be found in \cite[\lemmaInRef{1}\pageInRef{20}]{ndiaye:tel-01962450}.
}

\begin{proof}[Proof of~\eqref{theo: FY inequalities 2}]
	{
		\newcommand{\aSubGradientInTheProof}{\nabla\aFunction^*(\aDualVariable)}

		Assume that \(\aFunction\) is \(\convCoeff\)-strongly convex and let \(\aVariable,\aDualVariable\in\kR^\aDimension\).
		First note that since \(\aFunction\) is proper, closed and \(\convCoeff\)-strongly convex, we have from \cite[\itemInRef{(b)} of \theoremInRef{5.26}\pageInRef{123}]{Beck2017aa} that \(\aFunction^*\) is \(\kinv{\convCoeff}\)-Lipschitz smooth over \(\kR^\obsdim\).
		In particular, \(\domain(\aFunction^*)=\kR^\aDimension\) and \(\aFunction^*\) is differentiable at any \(\aDualVariable\in\kR^d\), that is
		\begin{equation}\label{eq:def z}
			\subDiff\aFunction^*(\aDualVariable) =
			\{
			\nabla \aFunction^*(\aDualVariable)
			\}.
		\end{equation}
		Second, if \(\aVariable\notin\domain(\aFunction)\), then the left-hand side of~\eqref{theo: FY inequalities 2} is infinite and the inequality trivially holds true since the right-hand side is finite.
		We conclude the proof by showing that~\eqref{theo: FY inequalities 2} is also valid for \(\aVariable\in\domain(\aFunction)\).
		Using the fact that \(\aFunction\) is proper, closed and convex, we obtain from \Cref{lemma:subdifferential inversion} with the pair \((\nabla\aFunction^*(\aDualVariable), \aDualVariable)\) that
		\begin{equation}
			\aDualVariable\in\partial\aFunction(\aSubGradientInTheProof)
			,
		\end{equation}
		and from \cite[\propositionInRef{16.4.(i)}]{Bauschke2017}, 
		\begin{equation}
			\nabla \aFunction^*(\aDualVariable)\in \domain(\aFunction)
			.
		\end{equation}
		Invoking the first-order characterization of \(\convCoeff\)-strong convexity of \(\aFunction\) at \(\nabla \aFunction^*(\aDualVariable)\in \domain(\aFunction)\) (see \cite[\theoremInRef{5.24}]{Beck2017aa}) then leads to 
		\begin{equation} \label{eq:proof fenchel-young strong convexity:applying strong convexity}
			\aFunction(\aVariable)
			\geq
			\aFunction(\aSubGradientInTheProof) + \dotp{\aDualVariable}{\aVariable - \aSubGradientInTheProof}
			+ \tfrac{\convCoeff}{2}\kvvbar{\aVariable - \aSubGradientInTheProof}^2.
		\end{equation}
		Finally, considering the standard Fenchel-Young inequality \eqref{theo: FY inequalities 1} with \(\aVariable=\aSubGradientInTheProof\) and using \eqref{eq:def z}, we have that the following equality holds: 
		\begin{equation}\label{eq:equality case FY}
			\aFunction(\aSubGradientInTheProof) + \aFunction^*(\aDualVariable) = \dotp{\aDualVariable}{\aSubGradientInTheProof}
			.
		\end{equation}
		We obtain the desired result~\eqref{theo: FY inequalities 2} by re-injecting~\eqref{eq:equality case FY} into~\eqref{eq:proof fenchel-young strong convexity:applying strong convexity}.
	}
\end{proof}
\vspace{0.2cm}

\iftoggle{pagebreakSection}{
	\vfill
	\pagebreak	
}{
}

%% file: manuscript_B_proof_safeness.tex

\section{Proofs related to construction of the \FBI{} framework}\label{app:proofs FBI}

\subsection{Proof of \texorpdfstring{\Cref{theo: holder ball}}{Section~\ref{prop:gap inequality}}} \label{subsec:proof gap inequality}

We first notice that our result in \Cref{theo: holder ball} is an equivalent rewriting of the following proposition:
\vspace{0.2cm}

\begin{prop} \label{prop:gap inequality}
	If hypotheses \hypref{H2}-\hypref{H3} hold true then the following inequality  is satisfied for  any \((\pv, \dv) \in \kR^\pvdim\times\kR^\dvdim\):
	\begin{equation} \label{eq:subsec:FBI inequality V-GAP}
		\normTwo{\dvopt - \dv}^2 + \normTwo{\dvopt + \nabla\loss(\dicomat\pv)}^2
		\leq
		\frac{2\gapfunc(\pv, \dv)}{\convCoeff}
		. \vspace{0.2cm}
	\end{equation}
\end{prop}
In the rest of this section, we thus concentrate on the proof of \Cref{prop:gap inequality}.
Our arguments leverage 
the following lemma whose proof is postponed to \Cref{proof:r_x GAP inequality}:\vspace{0.2cm}
\begin{lemm} \label{lemma:u - r_x GAP inequality}
	If hypotheses \hypref{H2}-\hypref{H3} hold true, then
	the following inequality is satisfied for any \((\pv,\dv)\in\kR^\pvdim\times\kR^\dvdim\):
	\begin{equation} \label{eq:lemma:u - r_x GAP inequality}
		\norm{\dv + \nabla\loss(\dicomat\pv)}_2^2
		\leq
		\frac{
			2\gapfunc(\pv,\dv)
		}{
			\convCoeff
		}
		. \vspace{0.2cm}
	\end{equation}
\end{lemm}

\Cref{prop:gap inequality} can be proved by applying \Cref{lemma:u - r_x GAP inequality} for two different choices of couple \((\pv,\dv)\).
We remind the reader that: {\textit{i)} we assume that \eqref{eq:primal problem (fuzzy version)} admits (at least) one minimizer \(\pvopt\); \textit{ii)} there exists a unique maximizer \(\dvopt\) to \eqref{def:dual solution} under \hypref{H2}-\hypref{H3}, see \Cref{lemma:strong duality}.} 
A first application of \Cref{lemma:u - r_x GAP inequality} with \((\pvopt,\dv)\) then leads to
\begin{equation} 
	\norm{\dv + \nabla\loss(\dicomat\pvopt)}_2^2
	\leq
	\frac{
		2(\pfunc(\pvopt) - \dfunc(\dv))
	}{
		\convCoeff
	}
	.
\end{equation}
Since \Cref{lemma:strong duality} also ensures that strong duality holds, we have \(\pfunc(\pvopt)=\dfunc(\dvopt)\) and \(\dvopt=-\nabla\loss(\dicomat\pvopt)\). Therefore, the previous inequality can also be rewritten as
\begin{equation} \label{eq:proof gap inequality:first application lemma}
	\norm{\dv - \dvopt}_2^2
	\leq
	\frac{
		2(\dfunc(\dvopt) - \dfunc(\dv))
	}{
		\convCoeff
	}
	.
\end{equation}
A second application of \Cref{lemma:u - r_x GAP inequality} with the pair \((\pv,\dvopt)\) yields
\begin{equation} \label{eq:proof gap inequality:second application lemma}
	\norm{\dvopt + \nabla\loss(\dicomat\pv)}_2^2
	\leq
	\frac{
		2(\pfunc(\pv) - \dfunc(\dvopt))
	}{
		\convCoeff
	}
	.
\end{equation}
Summing up \eqref{eq:proof gap inequality:first application lemma} and \eqref{eq:proof gap inequality:second application lemma} leads to the desired result~\eqref{eq:subsec:FBI inequality V-GAP}.


\vspace*{.1cm}

\subsection{Proof of \texorpdfstring{\cref{lemma:u - r_x GAP inequality}}{Lemma~\ref{lemma:u - r_x GAP inequality}}} \label{proof:r_x GAP inequality}
{
	If \((\pv,\dv)\notin \domain(\pfunc)\times \domain(-\dfunc)\), then \eqref{eq:lemma:u - r_x GAP inequality} 
	is trivially satisfied since the left-hand side is finite whereas the right-hand side is equal to \(+\infty\).
	In the rest of the proof, we thus assume that \((\pv,\dv)\in \domain(\pfunc)\times \domain(-\dfunc)\).

	As discussed in \Cref{sec:optimization framework}, hypotheses \hypref{H2}-\hypref{H3} imply that \(\loss^*\) is proper, closed and \(\convCoeff\)-strongly convex.
	Applying the refined Fenchel-Young inequality~\eqref{theo: FY inequalities 2} 
	with \(\aFunction=\loss^*\), \(\aVariable=-\dv\) and \(\aDualVariable=\dicomat\pv\) then leads to
	\begin{equation}\label{eq:proof:lemma:FBI:refined FY0}
		\loss^*(-\dv) + \loss^{**}(\dicomat\pv)
		\geq
		\dotp{-\dv}{\dicomat\pv}
		+ \frac{\convCoeff}{2}\norm{\dv + \nabla\loss^{**}(\dicomat\pv)}_2^2
		.
	\end{equation}
	Using \cite[\theoremInRef{4.8}\pageInRef{89}]{Beck2017aa}, we have that \(\loss^{**}=\loss\) since \(\loss\) is proper, closed and convex by \hypref{H2}. Therefore, \eqref{eq:proof:lemma:FBI:refined FY0} can be equivalently rewritten as:
	\begin{equation}\label{eq:proof:lemma:FBI:refined FY}
		\loss(\dicomat\pv) + \loss^*(-\dv)
		\geq
		\dotp{-\dv}{\dicomat\pv}
		+ \frac{\convCoeff}{2}\norm{\dv + \nabla\loss(\dicomat\pv)}_2^2
		.
	\end{equation}
	In order to conclude the proof, we need to add \(\reg(\pv) + \reg^*(\ktranspose{\dicomat}\dv)\) to both sides of this inequality.
	Prior to this operation, we have nevertheless to ensure that \(\reg(\pv) + \reg^*(\ktranspose{\dicomat}\dv)<+\infty\). To that end, we notice that \(\loss(\pv)>-\infty\) and \(\loss^*(-\dv)>-\infty\) since \(\loss\) and \(\loss^*\) are proper. Hence,
	\begin{align}
		\pfunc(\pv) < +\infty
		 & \implies
		\reg(\pv)<+\infty \\
		-\dfunc(\dv) < +\infty
		 & \implies
		\reg^*(\ktranspose{\dicomat}\dv)<+\infty
		.
	\end{align}
	Since we assume that \((\pv,\dv)\in \domain(\pfunc)\times \domain(-\dfunc)\), the left-hand sides of these implications are satisfied, so that \(\reg(\pv) + \reg^*(\ktranspose{\dicomat}\dv)<+\infty\).

	Adding \(\reg(\pv) + \reg^*(\ktranspose{\dicomat}\dv)\)
	to both sides of \eqref{eq:proof:lemma:FBI:refined FY} then leads to
	\begin{equation}
		\pfunc(\pv) - \dfunc(\dv)
		\geq
		\reg(\pv) + \reg^*(\ktranspose{\dicomat}\dv) - \dotp{\dv}{\dicomat\pv}
		+ \frac{\convCoeff}{2}\norm{\dv + \nabla \loss(\dicomat\pv)}_2^2
		.
	\end{equation}
	Finally, since \(\reg\) is proper, closed and convex from \hypref{H2}, we can apply the Fenchel-Young inequality~\eqref{theo: FY inequalities 1} with \(\aFunction=\reg\), \(\aVariable=\pv\) and \(\aDualVariable=\ktranspose{\dicomat}\dv\) to obtain~\eqref{eq:lemma:u - r_x GAP inequality}.\\
}

\iftoggle{pagebreakSection}{
	\vfill
	\pagebreak	
}{
}

%% file: manuscript_C_proof_connection.tex

\section{Proofs of the connections with existing results}
\label{app:proofs comparison}
This appendix gathers all proofs related to the comparison of the proposed \Holder ball with state-of-the-art safe regions presented in \Cref{sec:GAP ball} to \Cref{sec:SLORE and SFER balls}.

\subsection{Proof of \texorpdfstring{\Cref{lemma: connection between gap and divergence}}{Lemma~\ref{lemma: connection between gap and divergence}}}\label{proof:lemma: connection between gap and divergence}

Let \((\pv,\dv)\in \domain(\pfunc)\times\domain(-\dfunc)\).
Using the definition of the primal and dual cost functions, we obtain:
\begin{align}
	\gapfunc(\pv,\dv)
	 & \,=\,
	\loss(\dicomat\pv) + \reg(\pv) + \loss^*(-\dv) + \reg^*(\ktranspose{\dicomat}\dv)
	\nonumber \\
	 & \,=\,
	\loss(\dicomat\pv) + \loss^*(-\dv)
	+ \dotp{\dicomat\pv}{\dv}
	+\reg(\pv) + \reg^*(\ktranspose{\dicomat}\dv)
	- \dotp{\dicomat\pv}{\dv}
	\nonumber \\
	 & \,=\,
	\Fen(\pv,\dv) + \reg(\pv) + \reg^*(\ktranspose{\dicomat}\dv)
	- \dotp{\dicomat\pv}{\dv}
	.
	\nonumber
\end{align}
If \eqref{eq:condition couple u=-nabla f(Ax)} holds, we then have from \Cref{theo: FY inequalities} that
\begin{align}
	\reg(\pv) + \reg^*(\ktranspose{\dicomat}\dv)
	- \dotp{\dicomat\pv}{\dv}=0.
	\nonumber
\end{align}
This shows the first equality in \eqref{eq:connection between gap and divergence}.

The second inequality can be obtained by noticing that (from \Cref{theo: FY inequalities})
\begin{equation} \label{eq:proof fenchel bregman inequality:fenchel young equality case}
	\loss(\dicomat\pv) + \loss^*(\nabla\loss(\dicomat\pv)) = \dotp{\nabla\loss(\dicomat\pv)}{\dicomat\pv}
\end{equation}
since \(\loss\) is proper, convex and \(\partial\loss(\dicomat\pv)=\{\nabla\loss(\dicomat\pv)\}\).
Hence,
\begin{align}
	\Fen(\pv,\dv)
	\,=\, &
	\loss(\dicomat\pv) + \loss^*(-\dv) + \dotp{\dicomat\pv}{\dv}
	\nonumber               \\
	\,=\, &
	\loss^*(-\dv) - \loss^{*}(\nabla\loss(\dicomat\pv)) + \dotp{\dicomat\pv}{\dv + \nabla\loss(\dicomat\pv)}
	\nonumber               \\
	\,=\, & \Breg(\pv,\dv).
	\nonumber
\end{align}

\subsection{Proof of \texorpdfstring{\Cref{prop:FBI ball subset GAP ball}}{Proposition~\ref{prop:FBI ball subset GAP ball}}} \label{proof:FBI subset GAP}


Using \Cref{prop:gap inequality}, we have
\begin{align}\label{proof:eq:def FBI ball}
	\safeball({\HOLDERballc,\HOLDERballr})
	 & = \kset{\dv'\in\kR^\dvdim}{
		\normTwo{\dv' - \dv}^2 + \normTwo{\dv' + \nabla\loss(\dicomat\pv)}^2
		\leq
		\frac{2\gapfunc(\pv, \dv)}{\convCoeff}
	}
	,
\end{align}
whereas the definitions of the GAP and \(\pv-\)GAP balls (see~\eqref{eq:def GAP ball} and~\eqref{eq:xpball:defs}, respectively) lead to:
\begin{align}
	\safeball(\GAPballc,\GAPballr)   & =\kset{\dv'\in\kR^\dvdim}{
		\normTwo{\dv' - \dv}^2
		\leq
		\frac{2\gapfunc(\pv, \dv)}{\convCoeff}
	}
	\label{proof: def: GAP ball}                                  \\
	\safeball(\xGAPballc,\xGAPballr) & =\kset{\dv'\in\kR^\dvdim}{
		\normTwo{\dv' + \nabla\loss(\dicomat\pv)}^2
		\leq
		\frac{2\gapfunc(\pv, \dv)}{\convCoeff}
	}
	.
	\label{proof: def: xGAP ball}
\end{align}
Since the membership conditions in \eqref{proof: def: GAP ball} and \eqref{proof: def: xGAP ball} are relaxations of the inequality defining the \FBI{} ball in \eqref{proof:eq:def FBI ball}, inclusions \eqref{eq:FBI ball subset GAP ball B} and \eqref{eq:FBI ball subset xGAP ball} necessarily hold.

Finally, to prove strict inclusion it is then sufficient to note that
\begin{align}
	\HOLDERballr & \leq \tfrac{\gapfunc(\pv,\dv)}{\convCoeff}< \GAPballr \nonumber  \\
	\HOLDERballr & \leq \tfrac{\gapfunc(\pv,\dv)}{\convCoeff}< \xGAPballr \nonumber
\end{align}
whenever \(\gapfunc(\pv,\dv)\neq 0\).

\subsection{Proof of \texorpdfstring{\Cref{cor: dynamic EDPP ball}}{Proposition~\ref{cor: dynamic EDPP ball}}} \label{subsec:proof connection EDPP ball}

We first note that the \FBI{} ball in \Cref{theo: holder ball} is well-defined.
Indeed, functions \(\loss\) and \(\reg\) in \eqref{eq:FNE:H1}-\eqref{eq:FNE:H2} are closed, proper and convex, so that \hypref{H2} holds.
Moreover, \hypref{H3} is verified with \(\convCoeff=1\).
For any feasible \((t\pv,\dv)\) with \(t\geq 0\), we thus have by definition:
\begin{align}
	\HOLDERballc(t)
	%
	 & = \tfrac{1}{2}(\dv+\obs - t\dicomat\pv)
	 \quad 
	 \label{proof:dynEDDP:eq:expr HOLDERc} 
	 \\
	\HOLDERballr^2(t) 
	 & =
	\gapfunc(t\pv, \dv) - \tfrac{1}{4}\normTwo{\dv -\obs + t \dicomat\pv}^2.\label{proof:dynEDDP:eq:expr HOLDERr}
\end{align}
Using the definitions of \(\loss\) and \(\reg\) in \eqref{eq:dynEDDP:H1}-\eqref{eq:dynEDDP:H2}, we note that the duality gap can be expressed as
\begin{align}
	\gapfunc(\pv,\dv)  \,=\, &
	\tfrac{1}{2}\normTwo{\obs - \dicomat\pv}^2+ \lambda\norm{\pv}  - \tfrac{1}{2}\normTwo{\obs}^2 + \tfrac{1}{2}\normTwo{\obs - \dv}^2
	\nonumber                  \\
	\,=\,                    &
	\lambda\norm{\pv} + \tfrac{1}{2}\normTwo{\obs - \dicomat\pv}^2 - \dotp{\obs}{\dv} + \tfrac{1}{2}\normTwo{\dv}^2
	\nonumber                  \\
	\,=\,                    &
	\lambda \norm{\pv} + \tfrac{1}{2}\normTwo{\dv - \obs + \dicomat\pv}^2 - \dotp{\dicomat\pv}{\dv}
	\nonumber
\end{align}
so that
\begin{align}
	\HOLDERballr^2(t)
	 & =
	\lambda t \norm{\pv} - t \dotp{\dicomat\pv}{\dv} + \tfrac{1}{4}\normTwo{\dv -\obs + t \dicomat\pv}^2. \label{proof:dynEDDP:eq:expr HOLDERr B}
\end{align}
Proving item \ref{item:dynamic EDPP ball:definition gamma} of \Cref{cor: dynamic EDPP ball} is equivalent to showing that the variable \(\dynEDPPballparam\) defined in \eqref{eq:EDPP:def definition gamma0} verifies 
\begin{equation}\label{proof:dynEDDP:goal 1}
	\dynEDPPballparam \in \kargmin_{t\geq 0}\ \HOLDERballr^2(t)
	.
\end{equation}
Since \(\HOLDERballr^2(t)\) is a convex function, it is sufficient to show that \(\dynEDPPballparam\) satisfies the problem's first-order optimality condition, \ie
\begin{equation}\label{proof:dynEDDP:eq:optimality condition}
	\forall t\geq 0:\ (\HOLDERballr^2(\dynEDPPballparam))' (t-\dynEDPPballparam)\geq 0
	,
\end{equation}
where
\begin{equation}
	(\HOLDERballr^2(t))' = \lambda \|\pv\|-\tfrac{1}{2} \dotp{\dicomat\pv}{\dv+\obs} + \tfrac{t}{2}\|\dicomat\pv\|_2^2.
\end{equation}
We distinguish between three cases.
First, if \(\dicomat\pv={\bf0}_\pvdim\), then \((\HOLDERballr^2(t))'=\lambda \|\pv\|\geq 0\) so that
\begin{equation}
	0 \in \kargmin_{t\geq 0}\HOLDERballr^2(t).
\end{equation}
In this case, the definition of \(\dynEDPPballparam\) in \eqref{eq:EDPP:def definition gamma0} also leads to \(\dynEDPPballparam=0\) (by using the conventions \(0/0=0\) and \(1/0=+\infty\)).
Second, if \(\dicomat\pv\neq{\bf0}_\pvdim\) and
\begin{equation}
	\tilde{t}\triangleq \frac{\dotp{\dicomat\pv}{\obs+\dv}-2\lambda\|\pv\|}{\|\dicomat\pv\|_2^2}\geq 0,
\end{equation}
we easily have that
\begin{equation}
	\tilde{t} \in \kargmin_{t\geq 0}\HOLDERballr^2(t)
\end{equation}
since \((\HOLDERballr^2(\tilde{t}))'=0\).
In this case, one deduces \(\dynEDPPballparam=\frac{\dotp{\dicomat \pv}{\obs+\dv}-2\lambda\|\pv\|}{\|\dicomat\pv\|_2^2}\).
Finally, if \(\dicomat\pv\neq{\bf0}_\pvdim\) and \(\tilde{t}<0\), we then have that \((\HOLDERballr^2(0))'\geq 0\) since \((\HOLDERballr^2(0))' = - \frac{\|\dicomat\pv\|_2^2}{2} \tilde{t}\) and \(\tilde{t}<0\).
In this case, \(0\) is a minimizer since it verifies \eqref{proof:dynEDDP:eq:optimality condition} and this corresponds again to the definition of \(\dynEDPPballparam\) in \eqref{eq:EDPP:def definition gamma0}.\\

Showing item \ref{item:dynamic EDPP ball:special case holder ball} of \Cref{cor: dynamic EDPP ball} is tantamount to showing that
\begin{align}
	\HOLDERballc(\dynEDPPballparam) & = \dynEDPPballc \label{proof:dynEDDP:goal 2a}  \\
	\HOLDERballr(\dynEDPPballparam) & = \dynEDPPballr \label{proof:dynEDDP:goal 2b}.
\end{align}
On the one hand, since item \ref{item:dynamic EDPP ball:definition gamma} of \Cref{cor: dynamic EDPP ball} is true, we directly have from the expression of \(\HOLDERballc(t)\) in \eqref{proof:dynEDDP:eq:expr HOLDERc} that \eqref{proof:dynEDDP:goal 2a} holds.
On the other hand, \eqref{proof:dynEDDP:goal 2b} can be shown by examining the following two cases.

If \(\dynEDPPballparam=0\), the equality in \eqref{proof:dynEDDP:goal 2b} follows directly from the definition \eqref{proof:dynEDDP:eq:expr HOLDERr B}.
If \(\dynEDPPballparam>0\), we have \((\HOLDERballr^2(\dynEDPPballparam))'=0\), \ie
\begin{align}
	\lambda \|\pv\| = \tfrac{1}{2} \dotp{\dicomat\pv}{\dv+\obs} - \tfrac{\dynEDPPballparam}{2}\|\dicomat\pv\|_2^2
	.
\end{align}
Plugging this equality into \eqref{proof:dynEDDP:eq:expr HOLDERr B} then leads to
\begin{align}
	\HOLDERballr^2(\dynEDPPballparam)
	 & =
	\tfrac{1}{2}\dotp{ \dynEDPPballparam\dicomat\pv }{ \obs - \dv }
	- \tfrac{1}{2} \normTwo{ \dynEDPPballparam \dicomat\pv }^2
	+ \tfrac{1}{4}\normTwo{\dynEDPPballparam\dicomat\pv - \obs + \dv}^2
	\nonumber                                                                         \\
	 & =\tfrac{1}{4}\|\obs-\dv\|_2^2-\tfrac{1}{4}\|\dynEDPPballparam\dicomat\pv\|_2^2
	\nonumber                                                                         \\
	 & = \dynEDPPballr^{2}
	. \nonumber
\end{align}

\subsection{Proof of \texorpdfstring{\Cref{prop:Sequential FNE ball is a special case of Holder ball}}{Proposition~\ref{prop:Sequential FNE ball is a special case of Holder ball}}} \label{subsec:proof connection FNE ball}

It is straightforward to see from the definition of \(\loss\) and \(\reg\) in \eqref{eq:FNE:H1}-\eqref{eq:FNE:H2} that \hypref{H2}-\hypref{H3} are satisfied with \(\convCoeff=1\).
The \FBI{} ball in \Cref{theo: holder ball} is therefore well-defined. We next show that the center and radius of the \FBI{} and FNE balls coincide.

First, using the definition of \(\loss\) in~\eqref{eq:FNE:H1}, we have
\begin{equation} \label{eq: expression grad FNE}
	\nabla\loss(\dicomat\pv) = -(\obs-\dicomat\pv)
	.
\end{equation}
Hence,
\begin{equation*}
	\HOLDERballc
	= \tfrac{1}{2}(\dv -\nabla\loss(\dicomat\pv))
	= \dv + \tfrac{1}{2}(\obs- \dicomat \pv - \dv )
	= \FNEballc
	.
\end{equation*}

Second, using the definition of \(\reg\) in \eqref{eq:FNE:H2} and \eqref{eq:subdiff norm l1}, it can be seen that condition ``\(\dotp{\dv}{\dicomat \pv} = \lambda\normOne{\pv}\)'' in \eqref{eq: hat opt 2} together with feasibility of \((\pv,\dv)\) is equivalent to ``\(\ktranspose{\dicomat}\dv = \partial\reg(\pv)\)'' in \eqref{eq:condition couple u=-nabla f(Ax)} so that \Cref{lemma: connection between gap and divergence} applies.
In particular, we have:
\begin{equation*}
	\gapfunc(\pv,\dv)
	=
	\Fen(\pv,\dv)
	= \loss(\dicomat\pv) + \loss^*(-\dv) + \dotp{\dv}{\dicomat\pv}
	.
\end{equation*}
Since
\begin{equation*}
	\loss^*(\dv)= \tfrac{1}{2}\|\dv\|_2^2+\dotp{\dv}{\obs},
\end{equation*}
the duality gap can thus also be written as
\begin{align}
	\gapfunc(\pv,\dv) & = \tfrac{1}{2}\|\obs-\dicomat\pv\|_2^2 + \tfrac{1}{2}\|\dv\|_2^2-\dotp{\dv}{\obs}+
	\dotp{\dv}{\dicomat\pv}\nonumber                                                                       \\
	                  & =\tfrac{1}{2}\|\obs-\dicomat\pv-\dv\|_2^2. \label{eq:expression GAP FNE}
\end{align}
Going back to the definition of the radius of the \FBI{} ball (with \(\convCoeff=1\)), we finally obtain:
\begin{align}
	\HOLDERballr^2
	 & = \gapfunc(\pv, \dv) - \tfrac{1}{4}\normTwo{\dv + \nabla\loss(\dicomat\pv)}^2
	\nonumber                                                                                 \\
	 & = \tfrac{1}{2}\|\obs-\dicomat\pv-\dv\|_2^2 - \tfrac{1}{4} \|\obs-\dicomat\pv-\dv\|_2^2
	\nonumber                                                                                 \\
	 & = \FNEballr^2
	,
	\nonumber
\end{align}
where we have used \eqref{eq: expression grad FNE} and \eqref{eq:expression GAP FNE} in the second equality.

\subsection{Proof of \texorpdfstring{\Cref{prop: connection to SAFE ball}}{Proposition~\ref{prop: connection to SAFE ball}}}
\label{subsec:proof:prop: connection to SAFE ball}
It is easy to see that the definitions of \(\loss\) and \(\reg\) in \eqref{eq:FNE:H1}-\eqref{eq:FNE:H2} verify \hypref{H2}-\hypref{H3} with \(\convCoeff=1\), so that
the \FBI{} ball in \Cref{theo: holder ball} is well-defined.

On the one hand, using the definition of \(\loss\) in \eqref{eq:FNE:H1} with \(\pv={\bf0}_\pvdim\), we have
\begin{equation}
	\nabla \loss(\dicomat\pv) = -\obs
	.
\end{equation}
On the other hand, noticing that the couple \(({\bf0}_\pvdim,\dv)\) verifies \eqref{eq:condition couple u=-nabla f(Ax)} and using the same reasoning as in the proof of \Cref{prop:Sequential FNE ball is a special case of Holder ball}, we obtain from \eqref{eq:expression GAP FNE}:
\begin{equation}
	\gapfunc(\pv,\dv) = \tfrac{1}{2}\|\obs-\dv\|_2^2.
\end{equation}
Finally, using \Cref{prop:gap inequality}, we have
\begin{equation}\label{proof:SAFE:eq:def FBI ball}
	\HOLDERball({\bf0}_\pvdim,\dv)
	= \kset{\dv'\in\kR^\dvdim}{
		\normTwo{\dv' - \dv}^2 + \normTwo{\dv' -\obs}^2
		\leq \|\obs-\dv\|_2^2
	}
	,
\end{equation}
whereas the SAFE ball is defined as
\begin{equation}\label{proof:SAFE:eq:def SAFE ball}
	\safeball({\SAFEballc,\SAFEballr})
	= \kset{\dv'\in\kR^\dvdim}{
		\normTwo{\dv' -\obs}^2
		\leq \|\obs-\dv\|_2^2
	}
	.
\end{equation}
Since the membership condition in \eqref{proof:SAFE:eq:def SAFE ball} is a relaxation of the inequality in \eqref{proof:SAFE:eq:def FBI ball}, inclusion \eqref{prop:eq:inclusion FBI in SAFE} holds.

\subsection{Proof of \texorpdfstring{\Cref{prop: connection to SLORE and SFER balls}}{Proposition~\ref{prop: connection to SLORE and SFER balls}}}
\label{subsec:proof around FBI inequality for logistic loss + ell_1}

It is easy to see from the definition of \(\loss\) and \(\reg\) in \eqref{eq:SLORE:H1}-\eqref{eq:SLORE:H2} that \hypref{H2}-\hypref{H3} hold with \(\convCoeff = 4\).
The \FBI{} ball in \Cref{theo: holder ball} is therefore well-defined.

We next show that the center and the radius of the \FBI{} and SFER balls coincide.
Since \((\pv, \dv)\) is defined as in \eqref{eq:def couple seq setup 1}-\eqref{eq:def couple seq setup 2}, we have from \Cref{lemma: generating primal-dual couple with rescaled problem} that this couple satisfies \eqref{eq:condition couple u=-nabla f(Ax)}. Using \Cref{lemma: connection between gap and divergence} then leads to
\begin{equation} \label{proof:SFER:gap=breg}
	\gapfunc(\pv,\dv) = \Breg(\pv,\dv)
	.
\end{equation}
Moreover, we also have from \eqref{eq:def couple seq setup 1}-\eqref{eq:def couple seq setup 2}:
\begin{equation}\label{proof:SFER:link gradf dv}
	\nabla \loss(\dicomat\pv) = -\gamma \dv.
\end{equation}
Using \eqref{proof:SFER:gap=breg}-\eqref{proof:SFER:link gradf dv}, we then easily find that
\begin{align}
	\HOLDERballc   & = \tfrac{1}{2}(\dv-\nabla \loss(\dicomat\pv))= \tfrac{1}{2}(1+\gamma)\dv=
	\seqFERballc\nonumber
	\\
	\HOLDERballr^2 & =  \tfrac{1}{4}\gapfunc(\pv,\dv)-\tfrac{1}{4}\|\dv+\nabla \loss(\dicomat\pv)\|_2^2\nonumber \\
	               & =\tfrac{1}{4}\Breg(\pv,\dv)-\tfrac{1}{4}\|(1-\gamma)\dv\|_2^2\nonumber                      \\
	               & =\seqFERballr^2.\nonumber
\end{align}
This shows the equality in \eqref{eq:connection FBI, SLORE, SFER}.
